\documentclass{article}
\usepackage[CJKbookmarks=true]{hyperref}
\usepackage{ctex}
\usepackage{amsmath}
\usepackage{amsfonts}
\usepackage{amsthm}
\usepackage{amssymb}
\usepackage{txfonts}
\usepackage{mathrsfs}
\usepackage{graphicx}
\usepackage[all]{xy}
\usepackage{CJK}
\usepackage{cases}

\begin{document}
\title{Radial prescribing scalar curvature on $RP^n$}
\author{\href{mailto:hongliu8808@163.com}{Liu Hong}}
\date{Mar 7, 2015}

\maketitle
\tableofcontents
\numberwithin{equation}{section}

\newtheorem*{de}{Definition}
\newtheorem{lem}{Lemma}
\newtheorem{cor}{Corollary}
\newtheorem{thm}{Theorem}
\newtheorem*{eg}{Example}
\newtheorem{pro}{Proposition}
\newtheorem{rmk}{Remark}
\ \\
\thispagestyle{headings}
\begin{abstract}
The study of radial prescribing scalar curvature of X.Xu and P.C.Yang [2] in 1993 showed a nonexistence result on $S^2$. Later in 1995, W.Chen and C.Li [2] generalized the nonexistence result to higher dimensions. G.Bianchi and E.Egnell [1] suggested that there may exist some non-negative smooth radial function which cannot be scalar curvature in the standard conformal class of $RP^n.$ However, in our paper, we prove that all smooth radial non-negative smooth functions which are positive on the pole could be prescribing scalar curvatures. We consider the quotient $\frac{v_\lambda}{V_\lambda}$ rather than the difference $v_\lambda-V_\lambda$ as in [1]. This trick yields a concise argument of the existence. Consequently, their counter examples stated in [1] cannot be true on $RP^n.$
\end{abstract}

\section{Introduction}
First of all, we use a simple argument to show that their examples cannot be true if $K(r)=K(\frac{1}{r}).$ So the counter examples constructed by G.Bianchi and E.Egnell in [1] do not exist on $RP^n.$

We regard $RP^n$ as the quotient space of the standard sphere $S^n$ under the antipodal map. $RP^n$ has a canonical conformal class inheriting from the standard class of $S^n.$

As in [1], by the north pole projection which projects the equator of $S^n$ onto the unit sphere of $\mathbb{R}^n$, the prescribing scalar curvature problem on $S^n$ is reduced to the following differential equation on $\mathbb{R}^n$
$$\Delta u+Ku^p=0\ \ ,\ \ p=\frac{n+2}{n-2}\ \ ,$$
where $K$ is a radial continuous function and $v^p=sgn(v)|v|^p$. We could write it down using polar coordinates as
$$v''(r)+\frac{n-1}{r}v'(r)+Kv^p(r)=0.$$
Due to the symmetry, we should have $v'(0)=0.$ So we consider the initial value problem of the ODE:
$$\left\{
  \begin{array}{ll}
    v''(r)+\frac{n-1}{r}v'(r)+Kv^p(r)=0, & \hbox{} \\
    v(0)=\lambda,\ v'(0)=0, & \hbox{}
  \end{array}
\right.
$$
Noticing that this initial value problem always has a short time solution, we write it as $v_\lambda.$
As we assume K is the prescribing scalar curvature on $RP^n$, $K(r)=K(\frac{1}{r})$ and u should be invariant under the Kelvin transformation. Namely,
$$v(r)=r^{2-n}v(\frac{1}{r})\ \ .$$
Thus, we need to solve the equation system
\begin{equation}
 \begin{cases}
  v''(r)+\frac{n-1}{r}v'+Kv^p=0, & \hbox{}\\
  v(0)=\lambda,\ v'(0)=0, \\
  v(r)=r^{2-n}v(\frac{1}{r}).
 \end{cases}
\end{equation}
When we take the derivative on $r=1,$ we get
$$(n-2)v(1)+2v'(1)=0$$
called gluing equation. So we can define the gluing function
$$G(\lambda):=(n-2)v_\lambda(1)+2v'_\lambda(1)\ \ .$$
We define the following quantities as in [1]:
$$\lambda_0:=\sup\{\alpha:v_\lambda(r)\ exists\ and\ is\ positive\ for\ all\ \lambda\in(0,\alpha),\ r\in[0,1]\},$$
$$\lambda_\infty:=\sup\{\alpha:v_\lambda(r)\ exists\ and\ is\ positive\ for\ all\ \lambda\in(\alpha,\infty),\ r\in[0,1]\}.$$
In this paper, we always assume $K(0)>0, K\geq0$ and K is continuous.

We restate some results in [1] first.

\begin{lem}
For any $\varepsilon\in[0,1],$
$$a(\varepsilon):=v_\lambda(\varepsilon)+\frac{\varepsilon v'_\lambda(\varepsilon)}{n-2}\ \ ,$$
$$b(\varepsilon):=-\frac{\varepsilon^{n-1}v'_\lambda(\varepsilon)}{n-2}\ \ ,$$
$$\gamma(\varepsilon):=|a(\varepsilon)|+|b(\varepsilon)\varepsilon^{n-2}|=|v_\lambda(\varepsilon)+\frac{\varepsilon v'_\lambda(\varepsilon)}{n-2}|+|\frac{\varepsilon v'_\lambda(\varepsilon)}{n-2}|\ \ .$$
Then we have the integral formulae
$$v_\lambda(r)=a(\varepsilon)+b(\varepsilon)r^{2-n}-\int_\varepsilon^rs^{1-n}\int_\varepsilon^st^{n-1}v_\lambda^pKdtds\ \ ,$$
$$v'_\lambda(r)=(2-n)b(\varepsilon)r^{1-n}-r^{1-n}\int_\varepsilon^rt^{n-1}v_\lambda^pKdt$$
and
$$\gamma(\varepsilon)<\frac{1}{2}(\frac{n}{\|K\|_\infty})^\frac{1}{p-1}\Longrightarrow|v_\lambda(r)|<2\gamma(\varepsilon),\ \forall r\in[\varepsilon,1]\ \ .$$
\end{lem}
\begin{proof}
The integral formulae follow from the following identity
$$(r^{n-1}v'_\lambda)'=r^{n-1}(v''_\lambda+\frac{n-1}{r}v'_\lambda)=r^{n-1}(-Kv_\lambda^p)=-Kr^{n-1}v_\lambda^p\ \ .$$
Suppose there is some value $r\in[\varepsilon,1]$ such that $|v_\lambda(r)|\geq2\gamma(\varepsilon),$ then we could pick the smallest $r_0$ among such values. The integral formula for $v_\lambda(r)$ implies
$$2\gamma(\varepsilon)=\sup\limits_{r\in[\varepsilon,r_0]}|v_\lambda(r)|\leq\gamma(\varepsilon)+(2\gamma(\varepsilon))^p\|K\|_\infty\frac{r_0^2}{2n}\leq\gamma(\varepsilon)+(2\gamma(\varepsilon))^p\frac{\|K\|_\infty}{2n}<2\gamma(\varepsilon)$$
which is a contradiction.
\end{proof}

\begin{cor}
When $\lambda\downarrow0,$ we have the following uniform asymptotic formulae on $[0,1]$
$$v_\lambda(r)=\lambda-\lambda^p\int_0^rs^{1-n}\int_0^st^{n-1}Kdtds+O(\lambda^{2p-1})\ \ ,$$
$$v'_\lambda(r)=-\lambda^p\int_0^rt^{n-1}Kdt+O(\lambda^{2p-1})\ \ .$$
\end{cor}
\begin{proof}
Using the integral formula for $\varepsilon=0,$ we have
$$v_\lambda(r)=\lambda-\int_0^rs^{1-n}\int_0^st^{n-1}v_\lambda^pKdtds\ \ .$$
When $\lambda\downarrow0,$ we know $v_\lambda(r)<2\lambda, r\in[0,1],$ which implies $v_\lambda(r)=\lambda+O(\lambda^p)$ uniformly on $[0,1]$ by the integral formula. Thus $v_\lambda^p(r)=\lambda^p+O(\lambda^{2p-1})$ uniformly on $[0,1].$ We use this in the integral formula to obtain
$$v_\lambda(r)=\lambda-\lambda^p\int_0^rs^{1-n}\int_0^st^{n-1}Kdtds+O(\lambda^{2p-1})\ \ , uniformly\ on\ [0,1].$$
The asymptotic formula for $v'_\lambda$ is deduced analogously.
\end{proof}

\begin{cor}
$\lambda_0>0.$
\end{cor}

By Lemma 3.1 of [1] or our Corollary 1, we have when $0<\lambda<<1,$
$$(n-2)v_\lambda(1)+2v'_\lambda(1)=(n-2)\lambda+o(\lambda)>0\ \ .$$
If $K\in C^1, K(r)=K(0)+K_\rho r^\rho+o(r^\rho),K_\rho\neq0,\rho\in(\frac{n(n-2)}{n+2},n],$ by Lemma 3.4 in [1], we have $\lambda_\infty=\infty$ or
$$(n-2)v_\lambda(1)+2v'_\lambda(1)=[\frac{n(n-2)}{K(0)}]^\frac{n-2}{2}(2-n)\lambda^{-1}+o(\lambda^{-1})<0,\ \ when\ \lambda\uparrow\infty\ .$$

\begin{lem}
If $\exists\lambda>0$ s.t. $(n-2)v_\lambda(1)+2v'_\lambda(1)=0,$ then the equation has a positive solution on $[0,\infty)$.
\end{lem}
\begin{proof}
If $v_\lambda>0$ in [0,1], one could use the Kelvin transformation
$$v_\lambda(r)=r^{2-n}v_\lambda(\frac{1}{r})$$
to expand the solution from $[0,1]$ to $[0,\infty).$\\
Now we could assume the smallest $\lambda$ such that the gluing equation holds is $\lambda_1.$ If $\lambda_0=\infty,$ then $v_\lambda>0$ in [0,1] always holds.
If $\lambda_0<\infty,$ from the definition, we know
$$v_{\lambda_0}(r)\geq0,\ \ \forall r\in[0,1]\ \ .$$
Thus, by the integral formula
$$v_{\lambda_0}'(r)=-r^{1-n}\int_0^rt^{n-1}v_{\lambda_0}^pKdtds$$
we know $v_{\lambda_0}'(1)<0.$ Thus
$$(n-2)v_{\lambda_0}(1)+2v'_{\lambda_0}(1)<0\ \ .$$
Meanwhile, we know
$$(n-2)v_\lambda(1)+2v'_\lambda(1)=(n-2)\lambda+o(\lambda)>0,\ \ when\ \lambda\downarrow0\ \ .$$
Hence $0<\lambda_1<\lambda_0$ and $v_{\lambda_1}>0$ in [0,1] by the definition of $\lambda_1.$
\end{proof}

\begin{cor}
If $\lambda_0<\infty,$ then (1.1) has a positive solution on $[0,\infty)$.
\end{cor}
\begin{proof}
Since $v_{\lambda_0}$ is decreasing and non-negative from the continuous reliability on parameters, we know $v_{\lambda_0}$ exists on $[0,1]$. As in the proof of Lemma 2, $G(\lambda_0)<0$ and $G(\lambda)<0$ for small $\lambda.$ Thus $\exists\lambda_1\in(0,\lambda_0)$ such that $G(\lambda_1)=0$ and $v_{\lambda_1}$ is the desired solution.
\end{proof}

\begin{thm}
If $K\in C^1(0,1], K(r)=K(0)+K_\rho r^\rho+o(r^\rho),K_\rho\neq0,\rho\in(\frac{n(n-2)}{n+2},n],$ then a positive solution of (1.1) on $[0,\infty)$ exists.
\end{thm}
\begin{proof}
From Lemma 3.4 in [1],
$$(n-2)v_\lambda(1)+2v'_\lambda(1)=[\frac{n(n-2)}{K(0)}]^\frac{n-2}{2}(2-n)\lambda^{-1}+o(\lambda^{-1})<0,\ \ \lambda\uparrow\infty\ .$$
If $\lambda_\infty=\infty$, then $\lambda_0<\infty$. So the result follows from Corollary 3.\\
If
$$(n-2)v_\lambda(1)+2v'_\lambda(1)=[\frac{n(n-2)}{K(0)}]^\frac{n-2}{2}(2-n)\lambda^{-1}+o(\lambda^{-1})<0,\ \ \lambda\uparrow\infty\ ,$$
from
$$(n-2)v_\lambda(1)+2v'_\lambda(1)=(n-2)\lambda+o(\lambda)>0,\ \ when\ \lambda\downarrow0\ \ $$
we know that there exists one $\lambda>0$ such that $(n-2)v_\lambda(1)+2v'_\lambda(1)=0$. Thus the result follows from Lemma 1.
\end{proof}

\section{Main result}

One could check that the solution of
$$\left\{
  \begin{array}{ll}
    v''(r)+\frac{n-1}{r}v'+K(0)v^p=0, & \hbox{} \\
    v(0)=\lambda,\ v'(0)=0, & \hbox{}
  \end{array}
\right.
$$
is
$$V_\lambda(r)=\frac{\lambda}{[1+\frac{\lambda^{2\beta}K(0)}{n(n-2)}r^2]^\frac{1}{\beta}},\ \ \beta=\frac{p-1}{2}=\frac{2}{n-2}.$$

\begin{lem}
If $\exists \varepsilon>0$ such that $K\in C^\infty[0,\varepsilon)$ and $v_\lambda$ exists on $[0,1]$ for all $\lambda>0$, then $G(\lambda)>0$ for $\lambda$ large enough.
\end{lem}
\begin{proof}
Consider $T_\lambda:=\frac{v_\lambda}{V_\lambda}.$ Since $v_\lambda=V_\lambda T_\lambda,$ we know that $T_\lambda$ satisfies the equation
$$V_\lambda'' T_\lambda+(2V_\lambda'+\frac{n-1}{r}V_\lambda)T_\lambda'+(V_\lambda''+\frac{n-1}{r}V_\lambda')T_\lambda+KV_\lambda^pT_\lambda^p=0$$
$$\Longleftrightarrow V_\lambda'' T_\lambda+(2V_\lambda'+\frac{n-1}{r}V_\lambda)T_\lambda'-KV_\lambda^pT_\lambda+KV_\lambda^pT_\lambda^p=0$$
$$\Longleftrightarrow r^5T_\lambda''+[n-1-\frac{2K(0)}{n}\frac{\lambda^{2\beta}r^2}{1+\frac{K(0)}{n(n-2)}r^2\lambda^{2\beta}}]r^4T_\lambda'+K\frac{\lambda^{2\beta}r^2}{[1+\frac{K(0)}{n(n-2)}r^2\lambda^{2\beta}]^2}r(-T_\lambda+T_\lambda^p)=0.$$
Thus, when $\lambda\rightarrow+\infty,$ the ODEs uniformly converge to
$$r^5T_\infty''+(3-n)r^4T_\infty'=0\ \ .$$
This equation has all solutions of the form
$$T_\infty=Cr^{n-2}+D.$$
Now we need the assumption $K\in C^\infty[0,\varepsilon).$
Since the family of ODEs are degenerate at $r=0,$ we should consider the initial value data $P=\{T(0)=1,\ T^{(k)}(0)=0,\ k=1,2,\cdots\}.$ Being subject to the initial data P, all the ODEs have only one solution. As $v_\lambda(0)=V_\lambda(0), v_\lambda'(0)=V_\lambda'(0),$ we know from the equation that $v_\lambda^{(k)}(0)=V_\lambda^{(k)}(0)$ and hence $T_\lambda^{(k)}(0)=0,\ \forall k=1,2,\cdots$ for $\log T_\lambda=\log v_\lambda-\log V_\lambda.$ Thus the unique solution is $T_\lambda$ and for the limit equation, the unique solution is $T_\infty=1.$ Thus, we know that
$$T_\lambda\longrightarrow1\ \ and\ \ T'_\lambda\longrightarrow0,$$
$$\frac{v_\lambda'}{v_\lambda}-\frac{V_\lambda'}{V_\lambda}=\frac{T_\lambda'}{T_\lambda}\longrightarrow0.$$
Since
$$\lim_{\lambda\rightarrow\infty}\frac{V_\lambda'(1)}{V_\lambda(1)}=2-n\ \ ,$$
we have
$$\lim_{\lambda\rightarrow\infty}\frac{v_\lambda'(1)}{v_\lambda(1)}=2-n\ \ .$$
Consequently, for $\lambda$ large enough, we have $G(\lambda)=(n-2)v_\lambda(1)+2v_\lambda'(1)<0.$
\end{proof}

\begin{thm}
If $K(0)>0,K\geq0,$ K is continuous and $\exists \varepsilon>0$ s.t. $K\in C^\infty[0,\varepsilon)$, then the equation (1.1) has a positive solution on $[0,\infty)$.
\end{thm}
\begin{proof}
If $\lambda_0<\infty, $ we already know the existence from Corollary 3. So we only need to treat the case $\lambda_0=\infty$ which means $v_\lambda$ exists and is positive on $[0,1]$ for every $\lambda>0.$ In this case, Lemma 2 and Lemma 3 deduce the result.
\end{proof}

\section{Acknowledgement}
The author is grateful to Professor Xingwang Xu for useful discussions and support all along.

\end{document}